\documentclass[11pt]{article}
\usepackage{amsfonts}
\usepackage{amssymb,amsthm,color,graphicx,graphics}
\usepackage{amsmath,comment}
\usepackage{xcolor}
\newcommand{\ds}{\displaystyle}

\newcommand{\R}{\mathbb{R}}
\newcommand{\N}{\mathbb{N}}

\newcommand{\eps}{\varepsilon}

\newtheorem{theorem}{Theorem}
\newtheorem{lemma}[theorem]{Lemma}

\newtheorem{corollary}[theorem]{Corollary}
\theoremstyle{definition}
\newtheorem{example}[theorem]{Example}
\newtheorem{definition}[theorem]{Definition}

\newcommand{\Om}{\Omega}

\newcommand{\bp}{\begin{proof}}
\newcommand{\ep}{\end{proof}}

\date{1 July 2021}
\begin{document}
\title{Efficiency and localisation for the first Dirichlet eigenfunction}

\author{\renewcommand{\thefootnote}{\arabic{footnote}}
M. van den Berg\,
\footnotemark[1]
\\
\renewcommand{\thefootnote}{\arabic{footnote}}
F. Della Pietra\,
\footnotemark[2]
\\
\renewcommand{\thefootnote}{\arabic{footnote}}
G. di Blasio\,
\footnotemark[3]
\\
\renewcommand{\thefootnote}{\arabic{footnote}}
N. Gavitone\,
\footnotemark[2]
\\
}
\footnotetext[1]{ School of Mathematics, University of Bristol, Fry Building,
Woodland Road, Bristol BS8 1UG, United Kingdom. \texttt{mamvdb@bristol.ac.uk}}

\footnotetext[2]{
Dipartimento di Matematica e Applicazioni ``Renato Caccioppoli'',
Universit\`a degli Studi di Napoli Federico II,
Via Cintia, Monte S. Angelo, I-80126 Napoli, Italy.
\texttt{f.dellapietra@unina.it},\,\, \texttt{nunzia.gavitone@unina.it}}

\footnotetext[3]{
Dipartimento di Matematica e Fisica,
Universit\`a degli Studi della Campania, L. Vanvitelli,
Viale Lincoln, 5 81100 Caserta, Italy.
\texttt{giuseppina.diblasio@unicampania.it}}

\maketitle

\begin{abstract}\noindent
Bounds are obtained for the efficiency or mean to max ratio $E(\Omega)$ for the first Dirichlet eigenfunction (positive)
for open, connected sets $\Omega$ with finite measure in Euclidean space
$\R^m$. It is shown that (i) localisation implies vanishing efficiency, (ii) a vanishing upper bound for the efficiency implies localisation, (iii) localisation occurs for the first Dirichlet eigenfunctions for a wide class of elongating bounded, open, convex and planar sets,
(iv) if $\Omega_n$ is any quadrilateral with perpendicular diagonals of lengths $1$ and $n$ respectively, then the sequence of first Dirichlet eigenfunctions localises, and $E(\Omega_n)=O\big(n^{-2/3}\log n\big)$.
This disproves some claims in the literature. A key technical tool is the Feynman-Kac formula.
\vspace{5mm}
\vspace{5mm}
\end{abstract}

\section{Introduction\label{sec1}}
Let $\Omega$ be a non-empty open set in Euclidean space $\R^m,\, m\ge 2,$ with boundary
$\partial\Omega$, and finite measure $|\Omega|$. It is well known
that the spectrum of the Dirichlet Laplacian acting in $L^2(\Omega)$
is discrete, and consists of an increasing sequence of eigenvalues
\begin{equation*}
\lambda_1(\Omega)\le \lambda_2(\Omega)\le ...,
\end{equation*}
accumulating at infinity. We denote a corresponding orthonormal
basis of eigenfunctions by $\{u_{j,\Omega},j\in
\N\}$,
\begin{equation*}
 -\Delta u_{j,\Omega}=\lambda_j(\Omega)u_{j,\Omega},\, u_{j,\Omega} \in
H_0^1(\Omega).
\end{equation*}
If $\lambda(\Omega):=\lambda_1(\Omega)$ has
multiplicity $1$, then $u_{1,\Omega}$ is uniquely defined up to a sign. This is the case if $\Omega$ is connected, for example.
We then write and choose, $u_{\Omega}:=u_{1,\Omega}>0$.

The Rayleigh-Ritz variational principle asserts that
\begin{equation}\label{e3}
\lambda(\Omega)=\inf_{\varphi\in
H_{0}^{1}(\Omega)\setminus\{0\}}\frac{\ds\int_{\Omega}|\nabla
\varphi |^{2}}{\ds\int_{\Omega}\varphi^{2}}.
\end{equation}
The \textit{efficiency} or \textit{mean to max ratio} of $u_{\Omega}$ is defined by
\begin{equation}\label{e4}
E(\Omega)=\frac{\Vert u_{\Omega}\Vert_1}{|\Omega|\| u_{\Omega} \|_{\infty}},
\end{equation}
where $\| \cdot\|_p,\, 1\le p\le \infty$ denotes the standard $L^p(\Omega)$ norm.

The study of $E(\Omega)$ goes back to the pioneering results of
\cite{ps,sperb}. In Theorem 3 of \cite{ps}, it was shown that if $\Omega$ is bounded and convex then
\begin{equation}\label{e5}
E(\Omega)\le \frac{2}{\pi},
\end{equation}
with equality in \eqref{e5} if $\Omega$ is a bounded interval in $\R$. A non-linear version has been proved in \cite{DPBG} for the $p$-Laplacian with $1<p<\infty$.
More general results have been obtained in \cite{CG}. It follows from inequality (3) and the main theorem in that paper that if $\Omega$ is a bounded region in $\R^m$, then
\begin{equation*}
E(\Omega)\ge E(B) \frac{|B|}{|\Omega|} \left(\frac{\lambda(B)}{\lambda(\Omega)}\right)^{m/2},
\end{equation*}
where $B$ is a ball in $\R^m$.

Moreover, it was asserted in Table 1 in \cite{ps} that
$\frac{2}{\pi}$ is the limit of the efficiency of a thinning annulus in $\R^m$. The proof of this assertion (Theorem \ref{the5}) will be given in Section \ref{sec4} below. There we will also compute the efficiency for the equilateral triangle, the square, and the disc. These data support the conjectures that (i) the efficiency of a bounded, convex planar set is maximised by the disc, (ii) if $P_n\subset \R^2$ is a regular $n$-gon then $n\mapsto E(P_n)$ is increasing. We note that the efficiency for an arbitrarily long rectangle is $(2/\pi)^2\approx 0.4053$, whereas the efficiency of a disc is approximately $0.4317$.

Recently a connection has been established between localisation of eigenfunctions and an effective potential such as the inverse of the torsion function (see \cite{ADFJM1}). In a similar spirit, it has been shown in certain special cases, such as a bounded interval in $\R$ or a square in $\R^2$, that the eigenfunctions of the Schr\"odinger operator of Anderson type localise (see \cite{ADFJM2}, and \cite{DFJM}).

The first part of the definition below is very similar to the one in \cite{grebenkov}, (formula (7.1) for $p=1$).
\begin{definition}\label{def1}
Let $(\Omega_n)$ be a sequence of non-empty open sets in $\R^m$ with $|\Omega_n|<\infty$.
\begin{enumerate}\item[\textup{(i)}] We say that a sequence $\big(f_{n}\big)$ with $f_n\in L^2(\Omega_n),\,n\in \N$ and $\|f_{n}\|_2=1$ localises if there exists a sequence
of measurable sets $A_n\subset\Omega_n$ such that
\begin{equation}\label{e12}
\lim_{n\rightarrow\infty}\frac{|A_n|}{|\Omega_n|}=0,\,\,\, \lim_{n\rightarrow\infty}\int_{A_{n}}f^2_{n}=1.
\end{equation}
\item[\textup{(ii)}] We say that a sequence $\big(f_{n}\big)$ with $f_n\in L^{\infty}(\Omega_n),\,f_n\ge 0,\,\|f_n\|_{\infty}>0,\,n\in \N$ has vanishing efficiency if
\begin{equation*}
\lim_{n\rightarrow\infty}\frac{\|f_n\|_1}{|\Omega_n|\|f_n\|_{\infty}}=0.
\end{equation*}
\end{enumerate}
\end{definition}

We have the following elementary observations.
\begin{lemma}\label{lemA} If $\Omega$ is a non-empty open set with finite Lebesgue measure, and if $\|f\|_2=1,\,0<\|f\|_{\infty}<\infty,$ with $f\ge 0$, then
\begin{enumerate}
\item[\textup{(i)}]
\begin{equation}\label{e18}
|\Omega|^{-1}\| f\|_{\infty}^{-2}\le \frac{\|f\|_1}{|\Omega|\|f\|_{\infty}}\le |\Omega|^{-1/2}\| f\|_{\infty}^{-1}.
\end{equation}
\item[\textup{(ii)}]
\begin{equation}\label{e19}
\frac{\|f\|_1}{|\Omega|\|f\|_{\infty}}\le |\Omega|^{-1}\|f\|_1^2.
\end{equation}
\end{enumerate}
\end{lemma}

The proofs of \eqref{e18} and \eqref{e19} are immediate, since by Cauchy-Schwarz,
\begin{equation*}
1=\|f\|_2\le \|f\|_{\infty}\|f\|_1\le \|f\|_{\infty}|\Omega|^{1/2}.
\end{equation*}

\begin{lemma}\label{lem0}
For $n\in \N$, let $f_n\in L^2(\Omega_n)$ with $\|f_{n}\|_2=1, \, f_n\ge 0,$ and $|\Omega_n|<\infty$. Then $\big(f_{n}\big)$ localises if and only if
\begin{equation}\label{e13}
\lim_{n\rightarrow\infty}\frac{1}{|\Omega_n|}\|f_n\|_1^2=0.
\end{equation}
\end{lemma}

By \eqref{e19} we have that if $\big(f_{n}\big)$ is localising then the mean to max ratio of $f_n$ is vanishing as $n\rightarrow\infty$.
We were unable to prove that if $\big(u_{\Omega_n}\big)$ has vanishing efficiency then $\big(u_{\Omega_n}\big)$ localises.

Denote by $\rho(\Omega)=\sup\{\min\{|x-y|:y\in \partial \Omega\},\,x\in \Omega\}$ the inradius of $\Omega$, by
$\textup{diam}(\Omega)=\sup\{|x-y|:x\in\Omega,\,y\in \Omega\}$ the diameter of $\Omega$, and by $w(\Omega)$ the width of $\Omega$. For a measurable set $A$ in $\R^k$ with $k<m$ we denote its $k$-dimensional Lebesgue measure by $|A|_k$. The indicator function of a set $A$ is denoted by ${\bf 1}_A$. We define for $\nu\ge0$, $j_{\nu}$ to be the first positive zero of the Bessel function $J_{\nu}$.

Below we show that sets with small $E(\Omega)$ have small inradius, and large diameter.
\begin{theorem}\label{the3} For all open, connected $\Omega\subset\R^m$ with $0<|\Omega|<\infty$,
\begin{equation}\label{e14}
\frac{\rho(\Omega)}{|\Omega|^{1/m}}\le \bigg(\frac{ej_{(m-2)/2}^2}{2\pi m}\bigg)^{1/2}E(\Omega)^{1/m}.
\end{equation}
If $\Omega$ is open, planar, bounded, and convex, then
 \begin{equation}\label{e15}
\frac{\textup{diam}(\Omega)}{|\Omega|^{1/2}}\ge\bigg(\frac{\pi}{ej_0^2}\bigg)^{1/2}E(\Omega)^{-1/2}.
\end{equation}
\end{theorem}

It is straightforward to construct sequences $(\Omega_n)$ for which $(u_{\Omega_n})$ is localising and, as a consequence of Lemma \ref{lem0} and \eqref{e19}, have vanishing efficiency.
For example, let $\Omega_n$ be the disjoint union of one disc $B$ with radius $1$, and $4n$ discs with radii $1/2$.  All of the $L^2$ mass of the first eigenfunction of $\Omega_n$ is supported on $B$, with $|B|/|\Omega_n|=\frac{1}{n+1}$, which tends to $0$ as $n\rightarrow\infty$.

Theorem \ref{the4} below together with Lemmas \ref{lemA} and \ref{lem0}, imply localisation for a wide class of sequences $(u_{\Omega_n})$. We first introduce the necessary notation.
\begin{definition}\label{def2} Points in $\R^m$ will be denoted by a Cartesian pair $(x_1,x')$ with $x_1\in \R,\, x'\in \R^{m-1}$. If $\Omega$ is an open set in $\R^m$, then we define $\Omega(x_1)=\{x'\in \R^{m-1}:(x_1,x')\in \Omega\}$. If $\Omega(x_1)$ is open, bounded, and non-empty in $\R^{m-1}$, then we denote its first $(m-1)$-dimensional Dirichlet eigenvalue by $\mu(\Omega(x_1))$. We also put
$\Omega'=\cup_{x_1\in\R}\Omega(x_1)$. A set $\Omega\subset\R^m$ is horn-shaped if it is open, connected, $x_1>x_2> 0$ implies $\Omega(x_1)\subset \Omega(x_2)$, and $x_1< x_2< 0$ implies $\Omega(x_1)\subset \Omega(x_2)$.
\end{definition}

\begin{theorem}\label{the4} Let $\Omega\subset\R^m$ be horn-shaped with $|\Omega|<\infty$ and $ |\Omega'|_{m-1}<\infty.$ If $\lambda\ge \lambda(\Omega)$,
\begin{equation}\label{e83a}
\mu(\Omega')\ge (m-1)(\lambda-\mu(\Omega')),
\end{equation}
and if
\begin{equation}\label{e83aa}
\varepsilon\in (0,|\Omega|\mu(\Omega')^{m/2}],
\end{equation}
then
\begin{align}\label{e83}
\frac{1}{|\Omega|}\|u_{\Omega}\|_1^2\le 2\varepsilon&+\frac{2|\Omega'|_{m-1}}{|\Omega|}\big|\big\{x_1\in\R:\frac{\mu(\Omega(x_1/2))-\mu(\Omega')}{2(\lambda-\mu(\Omega'))}
\le\log\big(\varepsilon^{-1}|\Omega|\mu(\Omega(x_1/2))^{m/2}\big)\big\}\big|_1\nonumber \\ &+\frac{2^{5/2}|\Omega'|_{m-1}}{|\Omega|}(\lambda-\mu(\Omega'))^{-1/2}\big(\log\big(\varepsilon^{-1}|\Omega|\mu(\Omega')^{m/2}\big)\big)^{1/2}.
\end{align}
\end{theorem}
If $\Omega\subset \R^2$ is open, bounded and convex, then it is always possible to find an isometry of $\Om$ such that this isometric set is horn-shaped: let $p$ and $q$ be points on $\partial\Omega$ such that $|p-q|=w(\Omega)$,
and $p-q$ is perpendicular to the pair of straight parallel lines tangent to $\partial\Omega$ at both $p$ and $q$ which define the width $w(\Omega)$. That such a pair $p,q$ exists was shown for example in Theorem 1.5
in \cite{MvdB3}. Let $T_{p,q}(\Om)=\{x-\frac12(p+q):x\in \Om\}$ be the translation of $\Om$ which translates the midpoint of $p$ and $q$ to the origin. Let $\varphi$ be the angle between the positive $x_1$ axis and the unit
 vector $(p-q)/|p-q|$, and let $R_{\varphi}$ be rotation over an angle $\frac{\pi}{2}-\varphi$. Then $R_{\varphi}T_{p,q}(\Om)$ is isometric with $\Om$, horn-shaped,
\begin{equation*}
R_{\varphi}T_{p,q}(\Om)'=(-|p-q|/2,|p-q|/2),
\end{equation*}
and
\begin{equation}\label{ea2}
|R_{\varphi}T_{p,q}(\Om)'|_1=w(\Om).
\end{equation}
The points $p$ and $q$ need not be unique, and so this isometry need not be unique. However, the construction above always gives \eqref{ea2}. If $\Upsilon$ is an ellipse with semi axes $a_1$ and $a_2$ with $a_1> a_2$ then
$R_{\varphi}T_{p,q}(\Upsilon)=\{(x_1,x_2):\big(\frac{x_1}{a_1}\big)^2+\big(\frac{x_2}{a_2}\big)^2<1\}$, and $|R_{\varphi}T_{p,q}(\Upsilon)'|_1=w(\Upsilon).$ However, the ellipse $\tilde{\Upsilon}=\{(x_1,x_2):\big(\frac{x_1}{a_2}\big)^2+\big(\frac{x_2}{a_1}\big)^2<1\}$ is a horn-shaped isometry of $\Upsilon$ with $|\tilde{\Upsilon}'|_1>w(\Upsilon)$.

\begin{corollary}\label{cor1}
Let $\Omega\subset \R^2$ be a convex horn-shaped set. If $\lambda\ge \lambda(\Omega)$ and $\mu(\Omega')\ge \frac12\lambda,$
then for $\varepsilon\in (0,|\Omega|\mu(\Omega'))$,
\begin{align}\label{e114}
\frac{1}{|\Omega|}\|u_{\Omega}\|_1^2\le 2\varepsilon &+\frac{2|\Omega'|_{1}}{|\Omega|}\big|\big\{x_1\in\R:\frac{|\Omega'|^2_{1}\mu(\Omega(x_1/2))-\pi^2}{2(|\Omega'|^2_{1}\lambda-\pi^2)}
\le\log\big(4\pi^2\varepsilon^{-1}|\Omega'|^{-2}_{1}|\Omega|\big)\big\}\big|_1\nonumber \\ &+\frac{2^{5/2}|\Omega'|_1^2}{|\Omega|}(|\Omega'|^2_{1}\lambda-\pi^2)^{-1/2}\big(\log\big(\pi^2\varepsilon^{-1}|\Omega'|^{-2}_{1}|\Omega|\big)\big)^{1/2}.
\end{align}
\end{corollary}

\begin{example}\label{exa1}
If $(a_n),(b_n),\, n\in \N$ are sequences in $\R$ satisfying $a_n\in[0,1],b_n\in [0,n]$, and if $\Omega_n$ is the quadrilateral with vertices $(0,a_n),(0,-1+a_n),(b_n,0),(-n+b_n,0),$ then
\begin{equation}\label{e85}
\frac{1}{|\Omega_n|}\|u_{\Omega_n}\|_1^2=O\big(n^{-2/3}\log n\big),\, n\rightarrow\infty,
\end{equation}
and $(u_{\Omega_n})$ is localising.
\end{example}
\begin{example}\label{exa2}
Let $R_n\subset \R^2$ be the rhombus with vertices $(\frac{n}{2},0),(-\frac{n}{2},0),(0,\frac12),(0,-\frac12)$, and let $\Omega_n$ be an open subset of $R_n$ which contains the open triangle $T_n$ with vertices
$(\frac{n}{2},0),(0,\frac12),(0,-\frac12).$ Then $\Omega_n$ satisfies \eqref{e85}, and $(u_{\Omega_n})$ is localising.
\end{example}

It follows by scaling properties of both $u_{\Omega}$ and $|\Omega|$ that if
$\Omega$ is open and connected with $|\Omega|<\infty$, and if
$\alpha>0$, then
\begin{equation*}
E(\alpha\Omega)=E(\Omega),
\end{equation*}
where $\alpha\Omega$ is a homothety of $\Omega$ by a factor
$\alpha$. Similarly,
\begin{equation*}
\frac{1}{|\alpha\Omega_n|}\|u_{\alpha\Omega_n}\|_1^2=\frac{1}{|\Omega_n|}\|u_{\Omega_n}\|_1^2.
\end{equation*}
Example \ref{exa2} then implies that a sequence of suitable translations, rotations and hometheties of sectors $(S_n(r)),$ with
\begin{equation*}
S_n(r):=\left\{(\rho,\theta):0<\rho<r, 0<\theta<\pi/n\right\}
\end{equation*}
satisfies
\begin{equation*}
\frac{1}{|S_n(r)|}\|u_{S_n(r)}\|_1^2=O\big(n^{-2/3}\log n\big),\, n\rightarrow\infty,
\end{equation*}
and $(u_{S_n(r)})$ localises as $n\rightarrow\infty$.  This could have been obtained directly using separation of variables, Kapteyn's inequality, and extensive computations involving Bessel functions. See \cite{kro} for similar computations.

\begin{example}\label{exa3}
If $1\le\alpha<\infty$, $m=2,3,...$, and
\begin{equation*}
\Omega_{n,\alpha}=\big\{(x_1,x')\in \R^m:\big(2n^{-1}|x_1|\big)^{\alpha}+|x'|^{\alpha}<1\big\},\, n\in \N,
\end{equation*}
then
\begin{equation}\label{e85b}
\frac{1}{|\Omega_{n,\alpha}|}\|u_{\Omega_{n,\alpha}}\|_1^2=O\big(n^{-2/(\alpha+2)}(\log n)^{\max\{1/\alpha,1/2\}}\big),\, n\rightarrow\infty,
\end{equation}
and $(u_{\Omega_{n,\alpha}})$ is localising.
\end{example}

\begin{theorem}\label{the5}
If $R>0,\,\eps>0$, and
\begin{equation*}
\Omega_{R,R+\eps}=\{x\in \R^{m}\colon R<|x|<R+\eps\},
\end{equation*} then
\begin{equation}\label{e22}
\lim_{\eps\downarrow 0}\eps^{2}\lambda(\Omega_{R,R+\eps})=\pi^{2},
\end{equation}
and
\begin{equation}\label{e23}
\lim_{\eps\downarrow 0}E(\Omega_{R,R+\eps})=\frac{2}{\pi}.
\end{equation}
If $\triangle\subset \R^2$ is an equilateral triangle, then
\begin{equation}\label{e24}
E(\triangle)=\frac{2}{\pi\sqrt 3}.
\end{equation}
If $\square\subset \R^2$ is a rectangle, then
\begin{equation}\label{e25}
E(\square)=\frac{4}{\pi^2}.
\end{equation}
If $B\subset\R^2$ is a disc, then
\begin{equation}\label{e26}
E(B)\approx 0.6782\frac{2}{\pi}.
\end{equation}
\end{theorem}

Inequalities (6.9) in \cite{grebenkov}, and (4.7) in \cite{kuttler} state that for $\Omega$ open, bounded, planar, and convex,
\begin{equation}\label{e6}
u_{\Omega}(x)\le \min\{|x-y|:y\in\partial\Omega\}\frac{\lambda(\Omega)^{1/2}}{|\Omega|}\| u_{\Omega}\|_1,
\end{equation}
and both papers refer to \cite{ps} for details. However, no such inequality can be found in \cite{ps}. Inequality \eqref{e6} would, by first maximising its right-hand side over all $x\in\Omega$, and subsequently its left-hand side
over all $x\in \Omega$, imply that
\begin{equation}\label{e7}
\| u_{\Omega}\|_{\infty}\le \rho(\Omega)\frac{\lambda(\Omega)^{1/2}}{|\Omega|}\| u_{\Omega}\|_1.
\end{equation}
Since the Dirichlet eigenvalues are monotone in the domain, and $\Omega$ contains a disc of radius $\rho(\Omega)$,
\begin{equation*}
\lambda(\Omega)\le \frac{j_0^2}{\rho(\Omega)^2}.
\end{equation*}
This, by \eqref{e4} and \eqref{e7}, implies that for a bounded, planar convex set $\Omega$,
\begin{equation}\label{e9}
E(\Omega)\ge j_0^{-1}.
\end{equation}
Inequality \eqref{e7} was also quoted in formula (2.24) in \cite{gamara}. However, \eqref{e7} and \eqref{e9} cannot hold true. Example \ref{exa1} above implies that $\lim_{n\rightarrow \infty}E(\Omega_n)=0$ for a collection of sequences of convex quadrilaterals $(\Omega_n)$. This collection includes a sequence of rhombi with vertices $(n/2,0),(-n/2,0),(0,1/2),(0,-1/2)$. This contradicts \eqref{e9}.

This paper is organised as follows. The proofs of Lemma \ref{lem0}, and Theorem \ref{the3} are deferred to Section \ref{sec2} below. The proofs of Theorem \ref{the4}, Corollary \ref{cor1}, and Examples \ref{exa1}, \ref{exa2}, and \ref{exa3} will be given in Section \ref{sec3}. The proof of Theorem \ref{the5} will be given in Section \ref{sec4}.

\section{Proofs of Lemma \ref{lem0} and Theorem \ref{the3}\label{sec2}}

\noindent{\it Proof of Lemma \textup{\ref{lem0}.}}
First suppose \eqref{e13} holds. That is if
\begin{equation}\label{e65}
a_n=\frac{1}{|\Omega_n|}\|f_n\|_1^2,
\end{equation}
then
\begin{equation}\label{e65a}
\lim_{n\rightarrow\infty}a_n=0.
\end{equation}
Let $\alpha>0$, and define
\begin{equation*}
B_{n,\alpha}=\{x\in \Omega_n:f_{n}(x)>\alpha\}.
\end{equation*}
It follows that
\begin{equation*}
\int_{\Omega_n\setminus B_{n,\alpha}}f_{n}^2\le \alpha^2|\Omega_n\setminus B_{n,\alpha}|,
\end{equation*}
and
\begin{equation*}
\int_{ B_{n,\alpha}}f_{n}^2\ge 1- \alpha^2|\Omega_n\setminus B_{n,\alpha}|\ge 1- \alpha^2|\Omega_n| .
\end{equation*}
Furthermore,
\begin{equation}\label{e69}
\int_{B_{n,\alpha}}f_{n}\ge \alpha|B_{n,\alpha}|.
\end{equation}
It follows by \eqref{e65} and \eqref{e69} that
\begin{equation*}
|B_{n,\alpha}|\le \alpha^{-1}\int_{B_{n,\alpha}}f_{n}\le \alpha^{-1}\int_{\Omega_n}f_{n}\le \alpha^{-1}a_n^{1/2}|\Omega_n|^{1/2}.
\end{equation*}
We now choose
\begin{equation*}
\alpha=a_n^{1/4}|\Omega_n|^{-1/2},
\end{equation*}
and conclude that,
\begin{equation*}
\int_{B_{n,a_n^{1/4}|\Omega_n|^{-1/2}}}f_{n}^2\ge 1- a_n^{1/2},\, \frac{|B_{n,a_n^{1/4}|\Omega_n|^{-1/2}}|}{|\Omega_n|}\le a_n^{1/4}.
\end{equation*}
Set $A_n=B_{n,a_n^{1/4}|\Omega_n|^{-1/2}}$. Then $A_n$ satisfies \eqref{e12} by \eqref{e65a}.

Next suppose \eqref{e12} holds. Let $\varepsilon\in(0,1)$ be arbitrary. There exists $N_{\varepsilon}\in \N$ such that both
\begin{equation*}
\int_{\Omega_n\setminus A_n}f_n^2<\varepsilon,
\end{equation*}
and $|A_n|/|\Omega_n|<\varepsilon$. So for $n\ge N_{\varepsilon}$,
\begin{align*}
\frac{1}{|\Omega_n|}\|f_n\|_1^2&
=\frac{1}{|\Omega_n|}\bigg(\int_{A_{n}}f_{n}+\int_{\Omega_n\setminus A_{n}}f_{n}\bigg)^2\nonumber \\ &
\le \frac{2}{|\Omega_n|}\bigg(\bigg(\int_{A_{n}}f_{n}\bigg)^2+\bigg(\int_{\Omega_n\setminus A_{n}}f_{n}\bigg)^2\bigg)\nonumber \\ &
\le \frac{2}{|\Omega_n|}\bigg(|A_{n}|+|\Omega_n\setminus A_{n}|\int_{\Omega_n\setminus A_{n}}f_{n}^2\bigg)\nonumber \\ &
\le 2\bigg(\frac{|A_{n}|}{|\Omega_n|}+\varepsilon\bigg)\nonumber \\ &\le 4\varepsilon.
\end{align*}
This concludes the proof since $\varepsilon\in(0,1)$ was arbitrary.
\hspace*{\fill }$\square $

\vspace{5mm}

\noindent{\it Proof of Theorem \textup{\ref{the3}}.}
By Lemma 3.1 in \cite{EBD} we have, taking into account that the estimates there are for the Dirichlet Laplacian with an extra factor $\frac12$, that
\begin{equation}\label{e73}
\|u_{\Omega}\|_{\infty}^2\le \left(\frac{e}{2\pi m}\right)^{m/2}\lambda(\Omega)^{m/2}.
\end{equation}
Since $\Omega$ contains a ball with inradius $\rho(\Omega)$, we have by domain monotonicity
\begin{equation}\label{e74}
\lambda(\Omega)\le \frac{j_{(m-2)/2}^2}{\rho(\Omega)^2}.
\end{equation}
By \eqref{e73}, and \eqref{e74},
\begin{equation*}
\|u_{\Omega}\|_{\infty}^{-2}\ge\bigg(\frac{2\pi m}{ej^2_{(m-2)/2}}\bigg)^{m/2}{\rho(\Omega)^{m}},
\end{equation*}
and \eqref{e14} follows by \eqref{e18}.
By \cite{H} we have that for planar convex sets, $|\Omega|\le 2\, \textup{diam}(\Omega)\rho(\Omega)$.
This, together with \eqref{e14}, implies \eqref{e15}.
\hspace*{\fill }$\square $

\section{Proofs of Theorem \ref{the4}, Corollary \ref{cor1}, and Examples \ref{exa1}, \ref{exa2}, \ref{exa3} \label{sec3}}

To prove Theorem \ref{the4} we proceed via a number of lemmas.
\begin{lemma}\label{lem1} If $\Omega$ is an open set with $|\Omega|<\infty$ and if $\Vert u_{\Omega}\Vert_2=1$, then for any $\varepsilon>0$,
\begin{equation}\label{e79}
\frac{1}{|\Omega|}\|u_{\Omega}\|_1^2\le 2\varepsilon^2|\Omega|+\frac{2}{|\Omega|}|\{x\in \Omega:u_{\Omega}(x)>\varepsilon\}|.
\end{equation}
\end{lemma}
\begin{proof}
Let
\begin{equation*}
\Omega^{\varepsilon}=\{x\in \Omega:u_{\Omega}\le\varepsilon\}.
\end{equation*}
We have by Cauchy-Schwarz that
\begin{align*}
\frac{1}{|\Omega|}\|u_{\Omega}\|_1^2&
=\frac{1}{|\Omega|}\bigg(\int_{\Omega^{\varepsilon}}u_{\Omega}+\int_{\Omega\setminus\Omega^{\varepsilon}}u_{\Omega}\bigg)^2\nonumber \\ &\le
\frac{2}{|\Omega|}\bigg(\bigg(\int_{\Omega^{\varepsilon}}u_{\Omega}\bigg)^2+\bigg(\int_{\Omega\setminus\Omega^{\varepsilon}}u_{\Omega}\bigg)^2\bigg)\nonumber\\ &
\le \frac{2}{|\Omega|}\bigg(\varepsilon^2|\Omega^{\varepsilon}|^2+|\Omega\setminus\Omega^{\varepsilon}|\int_{\Omega\setminus\Omega^{\varepsilon}}u_{\Omega}^2\bigg)\nonumber \\ &
\le2\varepsilon^2|\Omega|+\frac{2}{|\Omega|}|\{x\in \Omega:u_{\Omega}(x)>\varepsilon\}|.
\end{align*}
\end{proof}

For a non-empty open set $\Omega\subset \R^m$, we denote by $p_{\Omega}(x,y;t),\, x\in \Omega,y\in \Omega,\,t>0$ its Dirichlet heat kernel.
\begin{lemma}\label{lem2} If $\Omega$ is an open set in $\R^m$ with $0<|\Omega|<\infty$, then
\begin{equation}\label{e81}
p_{\Omega}(x,x;t)\le \bigg(\frac{e}{2\pi m}\bigg)^{m/2}\lambda(\Omega)^{m/2}e^{-t\lambda(\Omega)},\, t\ge \frac{m}{2\lambda(\Omega)}.
\end{equation}
\end{lemma}
\begin{proof}
Since $|\Omega|<\infty$, $p_{\Omega}(x,y;t)$ has an $L^2(\Omega)$ eigenfunction expansion given by
\begin{align}\label{e82}
\sum_{j=1}^{\infty}e^{-t\lambda_j(\Omega)}u_{j,\Omega}^2(x)=p_{\Omega}(x,x;t).
\end{align}
It follows from \eqref{e82} that for $\alpha\in [0,1)$,
\begin{align}\label{e86}
p_{\Omega}(x,x;t)&=\sum_{j=1}^{\infty}e^{-(\alpha+1-\alpha) t\lambda_j(\Omega)}u_{j,\Omega}^2(x)\nonumber \\ &\le e^{-\alpha t\lambda(\Omega)}\sum_{j=1}^{\infty}e^{-(1-\alpha) t\lambda_j(\Omega)}u_{j,\Omega}^2(x)\nonumber \\ &=e^{-\alpha t\lambda(\Omega)}p_{\Omega}(x,x;(1-\alpha)t)\nonumber \\ &\le
e^{-\alpha t\lambda(\Omega)}p_{\R^m}(x,x;(1-\alpha)t)\nonumber \\ &= e^{-\alpha t\lambda(\Omega)}(4\pi(1-\alpha)t)^{-m/2},
\end{align}
where we have used monotonicity of the Dirichlet heat kernel.
For $t\ge m/(2\lambda(\Omega))$ we choose  $\alpha$ as to optimise the right-hand side of \eqref{e86}. This yields,
\begin{equation*}
\alpha=1-\frac{m}{2t\lambda(\Omega)},
\end{equation*}
which in turn gives \eqref{e81}.
\end{proof}

The main idea in the proof of Theorem \ref{the4} is to use Brownian motion techniques to achieve an efficient way of separation
of variables for horn-shaped domains. These have been used extensively elsewhere. See for example \cite{BaMvdB}, and Lemma 7 in \cite{MvdB0}. If $\Omega(x_1)$ is open and non-empty then, following Definition \ref{def2}, we denote  corresponding Dirichlet heat kernel by $\pi_{\Omega(x_1)}(x',y';t),\, x'\in \Omega(x_1),\,y'\in \Omega(x_1),\, t>0$. We also put $\mu(\emptyset)=\infty, \pi_{\emptyset}(x',y';t)=0$.

\begin{lemma}\label{lem3}
Let $\Omega$ be a horn-shaped set in $\R^m$. If $x_1\in\R,x'\in \Omega(x_1)$, then
\begin{equation}\label{e88}
p_{\Omega}(x,x;t)\le (4\pi t)^{-1/2}\pi_{\Omega(x_1/2)}(x',x';t)+(4\pi t)^{-1/2}e^{-x_1^2/(4t)}\pi_{\Omega'}(x',x';t).
\end{equation}
\end{lemma}
\begin{proof}
The proof relies on the Feynman-Kac formula (\cite{DR}). We have that for any non-empty open set $\Omega$ in $\R^m$,
\begin{equation}\label{e89}
p_{\Omega}(x,y;t)=(4\pi t)^{-m/2}e^{-|x-y|^2/(4t)}\mathbb{P}\big(\cup_{0\le \tau\le t}x(\tau)\subset \Omega|\,\,x(0)=x,x(t)=y\big),
\end{equation}
where $\{x(\tau),\,0\le \tau\le t\}$ is a Brownian bridge on $\R^m$. The term $\mathbb{P}\big(\cup_{0\le \tau\le t}x(\tau)\subset \Omega|\,x(0)=x,x(t)=y\big)$ in \eqref{e89}
is the conditional probability that the Brownian bridge stays in $\Omega$, conditioned with $x(0)=x,\, x(t)=y$. We write $x(\tau)=(x_1(\tau),x'(\tau))$ with $x_1(0)=x_1,\, x_1(t)=y_1,\,x'(0)=x',\, x'(t)=y',$
where $\{x_1(\tau),0\le \tau\le t\}$, and $\{x'(\tau),0\le \tau\le t\}$ are independent Brownian bridges.

For $\xi>0$, and $x_1<\xi,\,y_1<\xi$, we have by the reflection principle,
\begin{equation*}
p_{(-\infty,\xi)}(x_1,y_1;t)=\frac{1}{(4\pi t)^{1/2}}\big(e^{-(x_1-y_1)^2/(4t)}-e^{-(2\xi-x_1-y_1)^2/(4t)}\big).
\end{equation*}
By \eqref{e89},
\begin{equation*}
\mathbb{P}\big(\max_{0\le \tau\le t} x_1(\tau)\le \xi|\,x_1(0)=x_1,\,x_1(t)=y_1\big)=1-e^{-(\xi-x_1)(\xi-y_1)/t},\, x_1<\xi,\,y_1<\xi.
\end{equation*}
For $x_1=y_1=0,\,\xi>0,$ we have
\begin{equation*}
\mathbb{P}\big(\max_{0\le \tau\le t} x_1(\tau)\le \xi|\,x_1(0)=x_1(t)=0\big)=1-e^{-\xi^2/t}.
\end{equation*}
We arrive at the well-known formula for the density of the maximum of a one-dimensional Brownian bridge,
\begin{equation}\label{e90a}
\mathbb{P}\big(\max_{0\le \tau\le t} x_1(\tau)\in d\xi|\,x_1(0)=x_1(t)=0\big)=\frac{2\xi}{t}e^{-\xi^2/t}{\bf1}_{[0,\infty)}(\xi)d\xi.
\end{equation}
We first consider the case $x_1>0$. By \eqref{e89}, and \eqref{e90a},
\begin{align}\label{e91}
p_{\Omega}(x,x;t)&=(4\pi t)^{-m/2}\mathbb{P}\big(\cup_{0\le \tau\le t}x(\tau)\subset \Omega|\,x(0)=x(t)=x\big)\nonumber \\ &
\le (4\pi t)^{-m/2}\int_0^{x_1/2}d\xi\frac{2\xi}{t}e^{-\xi^2/t}\mathbb{P}\big(\cup_{0\le \tau\le t}x'(\tau)\subset \Omega(x_1-\xi)|x'(0)=x'(t)=x'\big)\nonumber \\ &\hspace{4mm}+
(4\pi t)^{-m/2}\int_{x_1/2}^{\infty}d\xi\frac{2\xi}{t}e^{-\xi^2/t}\mathbb{P}\big(\cup_{0\le \tau\le t}x'(\tau)\subset \Omega'|x'(0)=x'(t)=x'\big)\nonumber \\ &\le
(4\pi t)^{-1/2}\int_0^{x_1/2}d\xi\frac{2\xi}{t}e^{-\xi^2/t}\pi_{\Omega(x_1/2)}(x',x';t)\nonumber \\ &\hspace{4mm}+(4\pi t)^{-1/2}\int_{x_1/2}^{\infty}d\xi\frac{2\xi}{t}e^{-\xi^2/t}\pi_{\Omega'}(x',x';t)\nonumber \\ &\le
(4\pi t)^{-1/2}\pi_{\Omega(x_1/2)}(x',x';t)+(4\pi t)^{-1/2}e^{-x_1^2/(4t)}\pi_{\Omega'}(x',x';t),
\end{align}
where we have used that $\Omega(x_1-\xi)\subset \Omega'$ for $\xi\ge x_1/2$ in the third line, and that $\Omega(x_1-\xi)\subset \Omega(x_1/2)$ for $\xi\in[0,x_1/2)$ in the fourth line. We next consider the case $x_1<0$.
By \eqref{e89}, and \eqref{e90a},
\begin{align}\label{e91a}
p_{\Omega}(x,x;t)&=(4\pi t)^{-m/2}\mathbb{P}\big(\cup_{0\le \tau\le t}x(\tau)\subset \Omega|\,x(0)=x(t)=x\big)\nonumber \\ &
\le (4\pi t)^{-m/2}\int_0^{|x_1|/2}d\xi\frac{2\xi}{t}e^{-\xi^2/t}\mathbb{P}\big(\cup_{0\le \tau\le t}x'(\tau)\subset \Omega(x_1+\xi)|\,x'(0)=x'(t)=x'\big)\nonumber \\ &\hspace{4mm}+
(4\pi t)^{-m/2}\int_{|x_1|/2}^{\infty}d\xi\frac{2\xi}{t}e^{-\xi^2/t}\mathbb{P}\big(\cup_{0\le \tau\le t}x'(\tau)\subset \Omega'|\,x'(0)=x'(t)=x'\big)\nonumber \\ &\le
(4\pi t)^{-1/2}\int_0^{|x_1|/2}d\xi\frac{2\xi}{t}e^{-\xi^2/t}\pi_{\Omega(x_1/2)}(x',x';t)\nonumber \\ &\hspace{4mm}+(4\pi t)^{-1/2}\int_{|x_1|/2}^{\infty}d\xi\frac{2\xi}{t}e^{-\xi^2/t}\pi_{\Omega'}(x',x';t)\nonumber \\ &\le
(4\pi t)^{-1/2}\pi_{\Omega(x_1/2)}(x',x';t)+(4\pi t)^{-1/2}e^{-x_1^2/(4t)}\pi_{\Omega'}(x',x';t),
\end{align}
where we have used that $\Omega(x_1+\xi)\subset \Omega'$ for $\xi\ge |x_1|/2$ in the third line, and that $\Omega(x_1+\xi)\subset \Omega(x_1/2)$ for $\xi\in[0,|x_1|/2)$ in the fourth line. Combining \eqref{e91} and \eqref{e91a} gives \eqref{e88}.
\end{proof}

\noindent{\it Proof of Theorem \textup{\ref{the4}}.}
We apply Lemma \ref{lem2} to the $(m-1)$-dimensional heat kernels $\pi_{\Omega(x_1/2)}$, and $\pi_{\Omega'}$ respectively, and obtain that for
\begin{equation}\label{e91b}
t\ge \frac{m-1}{2\mu(\Omega')}
\end{equation}
both
\begin{equation}\label{e93}
\pi_{\Omega(x_1/2)}(x',x';t)\le\bigg(\frac{e}{2\pi (m-1)}\bigg)^{(m-1)/2}\mu(\Omega(x_1/2))^{(m-1)/2}e^{-t\mu(\Omega(x_1/2))},
\end{equation}
and
\begin{align}\label{e94}
\pi_{\Omega'}(x'x';t)\le \bigg(\frac{e}{2\pi (m-1)}\bigg)^{(m-1)/2}\mu(\Omega')^{(m-1)/2}e^{-t\mu(\Omega')}.
\end{align}
Indeed, \eqref{e91b} implies $t\ge \frac{m-1}{2\mu(\Omega(x_1/2))}$ by domain monotonicity. For $t$ satisfying \eqref{e91b},
\begin{equation}\label{e95}
(4\pi t)^{-1/2}\le (\mu(\Omega')/(2\pi(m-1)))^{1/2},
\end{equation}
and we obtain, by Lemma \ref{lem3}, \eqref{e93}, \eqref{e94}, and \eqref{e95}, that for $t$ satisfying \eqref{e91b},
\begin{equation}\label{e96}
p_{\Omega}(x,x;t)\le  e^{-1/2}\bigg(\frac{e}{2\pi (m-1)}\bigg)^{m/2}\big(\mu(\Omega(x_1/2))^{m/2}e^{-t\mu(\Omega(x_1/2))}+\mu(\Omega')^{m/2}e^{-x_1^2/(4t)-t\mu(\Omega')}\big).
\end{equation}
Bounding the left-hand side of \eqref{e82} from below by $e^{-t\lambda}u_{\Omega}(x)^2$ we find by \eqref{e96} that if \eqref{e91b} holds, then
\begin{equation*}
u_{\Omega}(x)^2\le e^{-1/2}\bigg(\frac{e}{2\pi (m-1)}\bigg)^{m/2}\big(\mu(\Omega(x_1/2))^{m/2}e^{-t(\mu(\Omega(x_1/2))-\lambda)}+\mu(\Omega')^{m/2}e^{-x_1^2/(4t)-t(\mu(\Omega')-\lambda)}\big).
\end{equation*}
It follows that if \eqref{e91b} holds, then
\begin{align}\label{e98}
\{u_{\Omega}^2(x)\ge \varepsilon^2&\}\subset\big\{x\in\Omega:e^{-1/2}\bigg(\frac{e}{2\pi (m-1)}\bigg)^{m/2}\mu(\Omega(x_1/2))^{m/2}e^{-t(\mu(\Omega(x_1/2))-\lambda)}\ge\frac{\varepsilon^2}{2}\big\}\nonumber \\ &\hspace{5mm}\cup\big\{x\in\Omega: e^{-1/2}\bigg(\frac{e}{2\pi (m-1)}\bigg)^{m/2}\mu(\Omega')^{m/2}e^{-x_1^2/(4t)-t(\mu(\Omega')-\lambda)}\ge\frac{\varepsilon^2}{2}\big\}\nonumber \\ &=
\big\{x\in\Omega:2^{1/2}e^{-1/4}\big(\frac{e}{2\pi (m-1)}\bigg)^{m/4}\mu(\Omega(x_1/2))^{m/4}e^{-t(\mu(\Omega(x_1/2))-\lambda)/2}\ge\varepsilon\big\}\nonumber \\ &\hspace{5mm}\cup
\big\{x\in\Omega: 2^{1/2}e^{-1/4}\big(\frac{e}{2\pi (m-1)}\bigg)^{m/4}\mu(\Omega')^{m/4}e^{-x_1^2/(8t)-t(\mu(\Omega')-\lambda)/2}\ge\varepsilon\big\}\nonumber \\ &:=A_1\cup A_2,
\end{align}
with obvious notation. We choose
\begin{equation*}
t=(2(\lambda-\mu(\Omega')))^{-1},
\end{equation*}
and let
\begin{equation*}
\varepsilon\in(0,\mu(\Omega')^{m/4}].
\end{equation*}
Then the constraint on $t$ in \eqref{e91b} is satisfied for all $\Omega$ satisfying \eqref{e83a}. For the above choice of $t$ we have
\begin{equation*}
A_1\subset\big\{x\in \Omega:\frac{\mu(\Omega(x_1/2))-\mu(\Omega')}{4(\lambda-\mu(\Omega'))}<\log\big(\varepsilon^{-1}\mu(\Omega(x_1/2))^{m/4}\big)\big\},
\end{equation*}
\begin{equation}\label{e101}
|A_1|\le |\Omega'|_{m-1}\big|\big\{x_1\in\R:\frac{\mu(\Omega(x_1/2))-\mu(\Omega')}{4(\lambda-\mu(\Omega'))}<\log\big(\varepsilon^{-1}\mu(\Omega(x_1/2))^{m/4}\big)\big\}\big|_1,
\end{equation}
\begin{equation*}
A_2\subset\big\{x\in\Omega:x_1^2(\lambda-\mu(\Omega'))<4\log\big(\varepsilon^{-1}\mu(\Omega')^{m/4}\big)\},
\end{equation*}
and
\begin{equation}\label{e103}
|A_2|\le 4|\Omega'|_{m-1}(\lambda-\mu(\Omega'))^{-1/2}\big(\log\big(\varepsilon^{-1}\mu(\Omega')^{m/4}\big)\big)^{1/2}.
\end{equation}
By \eqref{e79}, \eqref{e98}, \eqref{e101}, and \eqref{e103}, we obtain
\begin{align*}
\frac{1}{|\Omega|}\|u_{\Omega}\|^2_1\le &2\varepsilon^2|\Omega| +\frac{2|\Omega'|_{m-1}}{|\Omega|}\big|\big\{x_1\in\R:\frac{\mu(\Omega(x_1/2))-\mu(\Omega')}{4(\lambda-\mu(\Omega'))}<\log\big(\varepsilon^{-1}\mu(\Omega(x_1/2))^{m/4}\big)\big\}\big|_1\nonumber\\&+
\frac{8|\Omega'|_{m-1}}{|\Omega|}(\lambda-\mu(\Omega'))^{-1/2}\big(\log\big(\varepsilon^{-1}\mu(\Omega')^{m/4}\big)\big)^{1/2}.
\end{align*}
Substitution of $\varepsilon^2|\Omega|=\varepsilon'$, and deleting the $'$ yields \eqref{e83} for all $\varepsilon$ satisfying \eqref{e83aa}.
\hspace*{\fill }$\square $

\vspace{10mm}
\noindent{\it Proof of Corollary \textup{\ref{cor1}}.}
Let
\begin{equation*}
x_1(\Omega)^+:=\sup\{x_1:\Omega(x_1)\ne \emptyset\}<\infty, \, x_1(\Omega)^-:=\inf\{x_1:\Omega(x_1)\ne \emptyset\}>-\infty.
\end{equation*}
Let $x_{\Omega}^+,x_{\Omega}^-$ be points of $\partial\Omega$ with $x_1$ coordinates $x_1(\Omega)^+$ and $x_1(\Omega)^-$ respectively.
By convexity $\Omega$ contains triangles with bases $\Omega'$ and vertices  $x_{\Omega}^+$ and  $x_{\Omega}^-$ respectively. Hence
for any $x=(x_1,x')\in \Omega$, $\frac12 x_1(\Omega)^-\le x_1/2\le \frac12 x_1(\Omega)^+,$ and $\Omega(x_1/2)$ contains a line segment with length at least $\frac12 |\Omega'|_1$. So
$$\mu(\Omega(x_1/2))\le 4\mu(\Omega')=\frac{4\pi^2}{|\Omega'|_1^2}.$$
This, together with \eqref{e83} for $m=2$, proves \eqref{e114}.
\hspace*{\fill }$\square $
\vspace{5mm}

P. Kr\"oger observed that one can get upper bounds for the first Dirichlet eigenvalue of the circular sector $S_n(r)$ with radius $r$ and opening angle $\pi/n$, which have the correct leading term by choosing an optimal rectangle inside the sector \cite{kro}. Similar observations were used in the proof of Theorem 1.5 in \cite{MvdB3}, and also in the proof of Theorem 1.3 in \cite{GJ}.\\

\noindent{\it Proof of Example \textup{\ref{exa1}}.}
Theorem 1.5 in \cite{MvdB3} implies the existence of a constant $c_1<\infty$ such that
\begin{equation}\label{e104}
\lambda(\Omega_{n})\le \pi^2+c_1n^{-2/3},\, n\in \N.
\end{equation}
We note that $\Omega_n$ is horn-shaped with respect to the coordinate system which defines it in Example \ref{exa1}. Note that $|\Omega'_n|_1=1$.  Straightforward computations show,
\begin{equation*}
\mu(\Omega_n(x_1))=\pi^2\bigg(1-\frac{x_1}{b_n}\bigg)^{-2},\, 0<x_1<b_n,
\end{equation*}
\begin{equation*}
\mu(\Omega_n(x_1))=\pi^2\bigg(1-\frac{|x_1|}{n-b_n}\bigg)^{-2},\, b_n-n<x_1<0,
\end{equation*}
\begin{equation}\label{e106}
\mu(\Omega_n(x_1/2))\ge\pi^2\bigg(1+\frac{x_1}{b_n}\bigg),\, 0<x_1<b_n,
\end{equation}
\begin{equation}\label{e106a}
\mu(\Omega_n(x_1/2))\ge\pi^2\bigg(1+\frac{|x_1|}{n-b_n}\bigg),\, b_n-n<x_1<0,
\end{equation}
and
\begin{equation}\label{e107}
|\Omega_n|=\frac{n}{2}.
\end{equation}
By \eqref{e104} we see that \eqref{e83a} holds for all
\begin{equation*}
n\ge N_{\Omega}:=\min\{n\in \N:n^{2/3}\ge \pi^{-2}c_1\}.
\end{equation*}
We obtain by Corollary \ref{cor1}, \eqref{e104}, \eqref{e106}, \eqref{e106a} and \eqref{e107} that for
\begin{equation}\label{e104a}
\lambda=\pi^2+c_1n^{-2/3},
\end{equation}
\begin{align}\label{e109}
\frac{2|\Omega_n'|_{1}}{|\Omega_n|}&\big|\big\{x_1\in\R:\frac{|\Omega_n'|^2_{1}\mu(\Omega_n(x_1/2))-\pi^2}{2(|\Omega_n'|^2_{1}\lambda-\pi^2)}
\le\log\big(4\pi^2\varepsilon^{-1}|\Omega_n'|^{-2}_{1}|\Omega_n|\big)\big\}\big|_1\nonumber \\ &\le 8\pi^{-2}c_1n^{-2/3}\log\big(2\pi^2\varepsilon^{-1}n\big).
\end{align}
The third term in the right-hand side of \eqref{e83} equals by \eqref{e104a},
\begin{equation}\label{e110}
\frac{2^{7/2}}{n^{2/3}}c_1^{-1/2}\big(\log\big(2^{-1}\pi^2\varepsilon^{-1}n\big)\big)^{1/2}.
\end{equation}
We find for $n\ge N_{\Omega}$, and $\varepsilon\in (0,2^{-1}\pi^2n]$ by \eqref{e109}, \eqref{e110}, and \eqref{e83},
\begin{align}\label{e111}
\frac{1}{|\Omega_n|}\|u_{\Omega_n}\|_1^2\le 2\varepsilon+8\pi^{-2}c_1n^{-2/3}\log\big(2\pi^2\varepsilon^{-1}n\big)+2^{7/2}c_1^{-1/2}n^{-2/3}\big(\log\big(2^{-1}\pi^2\varepsilon^{-1}n\big)\big)^{1/2}.
\end{align}
Choosing $\varepsilon=n^{-2/3}$ gives that the right-hand side of \eqref{e111} is $O\big(n^{-2/3}\log n\big).$ This implies localisation by Lemma \ref{lem0}, and \eqref{e85} follows by \eqref{e19} and \eqref{e111} for that choice of $\varepsilon$.
\hspace*{\fill }$\square $

\vspace{5mm}

\noindent{\it Proof of Example \textup{\ref{exa2}}.}
By choosing an optimal rectangle in $T_n$ one shows, similarly to \eqref{e104}, the existence of $c_3<\infty$ such that $\lambda(T_n) \le \pi^2+c_3n^{-2/3}$.
By domain monotonicity of the Dirichlet eigenvalues, and \eqref{e104a},
\begin{equation}\label{e116}
\lambda(R_n)\le \lambda(\Omega_n)\le \lambda(T_n)\le \pi^2+c_3n^{-2/3}.
\end{equation}
Furthermore we have,
\begin{equation*}
\frac{n}{4}=|T_n|\le |\Omega_n|\le |R_n|=\frac{n}{2},\,\,\mu(R_n')=\pi^2, |R_n'|_1=1.
\end{equation*}
By domain monotonicity of the Dirichlet heat kernels, we have for $\lambda\ge \lambda(\Omega_n)$,
\begin{align*}
e^{-t\lambda} u_{\Omega_n}(x)^2&\le e^{-t\lambda(\Omega_n)}u_{\Omega_n}(x)^2\le p_{\Omega_n}(x,x;t) \le p_{R_n}(x,x;t)\nonumber \\ &\le (4\pi t)^{-1/2}\pi_{R_n(x_1/2)}(x',x';t)+(4\pi t)^{-1/2}e^{-x_1^2/(4t)}\pi_{R_n'}(x',x';t).
\end{align*}
Adapting the proof of Theorem \ref{the4} from \eqref{e93} onwards, and adapting Corollary \ref{cor1}, gives for all $n$ sufficiently large, $\lambda\ge \lambda(\Omega_n)$, and $\varepsilon\le \frac{\pi^2n}{4}$,
\begin{align*}
\frac{1}{|\Omega_n|}\|u_{\Omega_n}\|_1^2\le & 2\varepsilon+\frac{2|R_n'|_{1}}{|\Omega_n|}\big|\big\{x_1\in\R:\frac{|R_n'|^2_{1}\mu(R_n(x_1/2))-\pi^2}{2(|R_n'|^2_{1}\lambda-\pi^2)}
\le\log\big(4\pi^2\varepsilon^{-1}|R_n'|^{-2}_{1}|\Omega_n|\big)\big\}\big|_1\nonumber \\ &\hspace{4mm}+\frac{2^{5/2}|R_n'|_1^2}{|\Omega_n|}(|R_n'|^2_{1}\lambda-\pi^2)^{-1/2}\big(\log\big(\pi^2\varepsilon^{-1}|R_n'|^{-2}_{1}|\Omega_n|\big)\big)^{1/2}\nonumber \\ \le &2\varepsilon +\frac{8}{n}\big|\big\{x_1\in\R:\frac{\mu(R_n'(x_1/2))-\pi^2}{2(\lambda-\pi^2)}
\le\log\big(2\pi^2\varepsilon^{-1}n\big)\big\}\big|_1\nonumber \\ &\hspace{4mm}+\frac{2^{9/2}}{n}(\lambda-\pi^2)^{-1/2}\big(\log\big(2^{-1}\pi^2\varepsilon^{-1}n\big)\big)^{1/2},
\end{align*}
where we have used \eqref{e116}. We now choose $\lambda= \pi^2+c_3n^{-2/3}$, and use \eqref{e106} and \eqref{e106a} with $b_n=\frac{n}{2}$. This gives
\begin{equation*}
\frac{1}{|\Omega_n|}\|u_{\Omega_n}\|_1^2\le2\varepsilon+16\pi^{-2}c_3n^{-2/3}\log\big(2\pi^2\varepsilon^{-1}n\big)+2^{9/2}c_3^{-1/2}n^{-2/3}\big(\log\big(2^{-1}\pi^2\varepsilon^{-1}n\big)\big)^{1/2}.
\end{equation*}
We choose $\varepsilon=n^{-2/3}$ which gives \eqref{e85}. This proves localisation by Lemma \ref{lem0}.
\hspace*{\fill }$\square $

\vspace{5mm}

\noindent{\it Proof of Example \textup{\ref{exa3}}.}
Theorem 1.5 in \cite{MvdB3} implies the existence of a constant $c(\alpha)\in (1,\infty)$ such that
\begin{equation}\label{e120}
\lambda(\Omega_{n,\alpha})\le j_{(m-2)/2}^2+c(\alpha)n^{-2\alpha/(\alpha+2)},\, n\in \N,
\end{equation}
where $\mu(\{x'\in\R^{m-1}:|x'|<1\})=j_{(m-2)/2}^2$.
For $-\frac n2<x_1<\frac n2,$ $\Omega(x_1)$ is an $(m-1)$-dimensional disc with radius $\big(1-(2|x_1|/n)^\alpha\big)^{1/\alpha}.$ Hence,
\begin{align}\label{e121}
\mu(\Omega(x_1/2))&=j_{(m-3)/2}^2\big(1-\big(n^{-1}|x_1|\big)^{\alpha}\big)^{-2/\alpha}\nonumber \\ &
\ge j_{(m-3)/2}^2\bigg(1+2\alpha^{-1}\big(n^{-1}|x_1|\big)^{\alpha}\big),
\end{align}
and
\begin{equation}\label{e121b}
|\Omega'_{n,\alpha}|_1=|\{x'\in\R^{m-1}:|x'|<1\}|_{m-1}=\omega_{m-1}, \, |\Omega_{n,\alpha}|=\omega_mn/2,
\end{equation}
and $\omega_m$ is the measure of the ball with radius $1$ in $\R^m$.
For $\varepsilon\in (0,2^{-1}\omega_mj_{(m-3)/2}^mn]$, $n$ sufficiently large, and $\lambda=j_{(m-2)/2}^2+c(\alpha)n^{-2\alpha/(\alpha+2)}\ge \lambda(\Omega_{n,\alpha})$, we have
\begin{align}\label{e122}
4\frac{\omega_{m-1}}{\omega_mn}\big|\big\{x_1\in\R:&\frac{\mu(\Omega_{n,\alpha}(x_1/2))-j_{(m-3)/2}^2}{2(\lambda-j_{(m-3)/2}^2)}
\le\log\big(2j_{(m-3)/2}^m\omega_m\varepsilon^{-1}n\big)\big\}\big|_1\nonumber \\ &\le \frac{4\omega_{m-1}}{\omega_m}\big(\alpha c(\alpha)/j_{(m-3)/2}^2\big)^{1/{\alpha}}n^{-2/(\alpha+2)}\big(\log\big(2j_{(m-3)/2}^m\omega_m\varepsilon^{-1}n\big)\big)^{1/\alpha}.
\end{align}
Similarly we find for $\varepsilon\in (0,2^{-1}\omega_mj_{(m-3)/2}^mn]$, and all $n$ sufficiently large,
\begin{align}\label{e123}
2^{7/2}\frac{\omega_{m-1}}{\omega_mn}(\lambda-j_{(m-3)/2}^2)^{-1/2}\big(\log\big(\pi^2\varepsilon^{-1}n\big)\big)^{1/2}\le 2^{7/2}c(\alpha)^{-1/2}n^{-2/(\alpha+2)}\big(\log\big(j_{(m-3)/2}^m\omega_m\varepsilon^{-1}n/2\big)\big)^{1/2}.
\end{align}
Choosing $\varepsilon=n^{-2/(\alpha+2)}$ gives \eqref{e85b} by Corollary \ref{cor1}, \eqref{e120}, \eqref{e121}, \eqref{e121b}, \eqref{e122} and \eqref{e123}. Lemma \ref{lem0} and \eqref{e85b} imply localisation.
\hspace*{\fill }$\square $

\section{Proof of Theorem \ref{the5}\label{sec4}}

{\it Proof of Theorem \textup{\ref{the5}.}} Choosing $\varphi(x)=\sin(\pi(|x|-R)/\varepsilon)$ as a test function in \eqref{e3}, we have that
\begin{align}\label{e27}
\lambda(\Omega_{R,R+\eps}) &\le
 \frac{\pi^{2}}{\eps^{2}}\dfrac{\ds\int_{R}^{R+\eps} \ds\cos^{2}\left(\pi(r-R)/\eps \right) r^{m-1}dr }{\ds\int_{R}^{R+\eps} \ds\sin^{2}\left(\pi(r-R)/\eps \right) r^{m-1}dr}\nonumber \\ &\le
\frac{\pi^{2}}{\eps^{2}} \left(\frac{R+\eps}{R}\right)^{m-1} \dfrac{\ds\int_{R}^{R+\eps} \ds\cos^{2}\left(\pi(r-R)/\eps \right) dr }{\ds\int_{R}^{R+\eps} \ds\sin^{2}\left(\pi(r-R)/\eps \right)dr}\nonumber \\ &=\frac{\pi^{2}}{\eps^{2}} \left(\frac{R+\eps}{R}\right)^{m-1}.
\end{align}
On the other hand, since the first Dirichlet eigenfunction of $\Omega_{R,R+\eps}$ is radial, $u_{\Omega_{R,R+\eps}}(x):=u(r)$, we have
\begin{align}\label{e28}
\lambda(\Omega_{R,R+\eps})&=\dfrac{\ds\int_{R}^{R+\eps}u'(r)^{2}r^{m-1}dr}{\ds\int_{R}^{R+\eps}u(r)^{2}r^{m-1}dr}\nonumber \\ &\ge \left(\frac{R}{R+\eps}\right)^{m-1}  \dfrac{\ds\int_{R}^{R+\eps}u'(r)^{2} dr}{\ds\int_{R}^{R+\eps}u(r)^{2} dr} \nonumber \\ & \ge \left(\frac{R}{R+\eps}\right)^{m-1}
\min_{v\in H_{0}^{1}(R,R+\eps)\setminus\{0\}}   \dfrac{\ds\int_{R}^{R+\eps}v'(r)^{2} dr}{\ds\int_{R}^{R+\eps}v(r)^{2} dr}\nonumber \\ & =\frac{\pi^{2}}{\eps^{2}}\left(\frac{R}{R+\eps}\right)^{m-1},
\end{align}
and \eqref{e22} follows from \eqref{e27} and \eqref{e28}.

To prove \eqref{e23} we consider the radial solution $\psi_{\eps}(|x|)=u_{\eps}(x)$ of
\begin{equation*}
-\Delta u_{\Omega_{R,R+\eps}}=\lambda(\Omega_{R,R+\eps})u_{\Omega_{R,R+\eps}},
\end{equation*}
with zero boundary condition, and $\|\psi_{\eps}\|_{\infty}=1$. The function $\psi_{\eps}$ satisfies
\begin{equation*}
\psi_{\eps}''+\dfrac{m-1}{r}\psi_{\eps}'+\lambda_{\eps}\psi_{\eps}=0\,\,\text{in}\,\,(R,R+\eps),
\end{equation*}
with boundary condition $\psi_{\eps}(R)=\psi_{\eps}(R+\eps)=0$, and normalisation $\|\psi_{\eps}\|_{\infty}=1$, where $\lambda_{\eps}=\lambda(\Omega_{R,R+\eps})$. Define
\begin{equation*}
\phi_{\eps}(t)=\psi_{\eps}(R+\eps t),\quad t\in (0,1).
\end{equation*}
Then $\phi_{\eps}$ satisfies
\begin{equation}\label{e32}
\left\{
\begin{array}{ll}
\phi_{\eps}''+\dfrac{(m-1)\eps}{R+\eps t}\phi_{\eps}'+\eps^{2}\lambda_{\eps}\phi_{\eps}=0&\,\text{in }(0,1),\\[.3cm]
\phi_{\eps}(0)=\phi_{\eps}(1)=0,\\[.3cm]
\|\phi_{\eps}\|_{\infty}=1.
\end{array}
\right.
\end{equation}
Integrating between the maximum point $t_{m}$ of $\phi$ and $t\in (0,1)$, we get that
\begin{align}\label{e33}
|\phi_{\eps}'(t)|&=\left|\int_{t_m}^{t}\left(\frac{(m-1)\eps\phi'_{\eps}(t)}{R+\eps t}+\eps^{2}\lambda_{\eps}\phi_{\eps}(t)\right)dt\right|\nonumber \\ &\le (m-1)\left( \frac{2\eps}{R}+\eps^{2}\lambda_{\eps}\right).
\end{align}
Hence $\phi_{\eps},\phi_{\eps}'$ are equibounded in $(0,1)$ and, by the Arzel\`a-Ascoli Theorem, $\phi_{\eps}$ converges uniformly, as $\eps\to 0^{+}$, to a continuous function $\phi(t)$ in $(0,1)$. From \eqref{e32}, \eqref{e33}, we also obtain equiboundedness of the second derivatives $\phi_{\eps}''$. Hence $\phi_{\eps}$ converges uniformly to $\phi$ in $C^{1}$. Moreover we obtain uniform convergence of the second derivatives $\phi_{\eps}''$. Passing to the limit in the equation, we infer that $\phi$ satisfies
\begin{equation*}
\left\{
\begin{array}{ll}
\phi''+\pi^{2}\phi=0&\text{in }(0,1),\\[.3cm]
\phi(0)=\phi(1)=0,\\[.3cm]
\|\phi\|_{\infty}=1.
\end{array}
\right.
\end{equation*}
Hence $\phi(t)=\sin (\pi t)$, and
\begin{equation}\label{e35}
\lim_{\eps\downarrow 0}\int_{[0,1]}\phi_{\eps}(t)dt=\int_{[0,1]}\phi(t)dt=\frac2\pi.
\end{equation}
So we obtain
\begin{align*}
E(\Omega_{R,R+\eps})&=|\Omega_{R,R+\eps}|^{-1}\int_{\Omega_{R,R+\eps}}\psi_{\eps}\nonumber \\ &
\ge \bigg(\frac{R}{R+\eps}\bigg)^{m-1}\int_{[0,1]}\phi_{\eps}(t)dt,
\end{align*}
and, by \eqref{e35},
\begin{equation*}
\liminf_{\eps\downarrow 0}E(\Omega_{R,R+\eps})\ge \frac2\pi.
\end{equation*}
Similarly we have
\begin{equation*}
E(\Omega_{R,R+\eps})\le \bigg(\frac{R+\eps}{R}\bigg)^{m-1}\int_{[0,1]}\phi_{\eps}(t)dt,
\end{equation*}
and, by \eqref{e35},
\begin{equation*}
\limsup_{\eps\downarrow 0}E(\Omega_{R,R+\eps})\le \frac2\pi.
\end{equation*}

To prove \eqref{e24} we consider an equilateral triangle $\triangle$ with vertices at $(0,0),(1,0),(\frac12,\frac12 \sqrt 3)$. The first Dirichlet eigenfunction is given by (formula (2.1) in \cite{S}),
\begin{equation*}
u_{\triangle}(x_1,x_2)=\sin\bigg(\frac{4\pi x_2}{\sqrt 3}\bigg)-\sin\bigg(2\pi\bigg(x_1+\frac{x_2}{\sqrt 3}\bigg)\bigg)+\sin\bigg(2\pi \bigg(x_1-\frac{x_2}{\sqrt 3}\bigg)\bigg).
\end{equation*}
We find that $|\triangle|=\frac{\sqrt 3}{4}$,
\begin{equation*}
\|u_{\triangle}\|_{\infty}=u(1/2,\sqrt 3/6)=\frac{3\sqrt 3}{2},
\end{equation*}
and
\begin{equation*}
\|u_{\triangle}\|_1=\int_{\triangle}u(x_1,x_2)dx_1\,dx_2=\frac{9}{4\pi\sqrt 3}.
\end{equation*}
This proves \eqref{e24}.

The efficiency of an interval is given by $\frac{2}{\pi}$. Formula \eqref{e25} follows by separation of variables. More generally if $\Omega_1$ and $\Omega_2$ are open and connected sets in
$\R^{m_1}$, and $\R^{m_2}$ respectively, and with finite measures
$|\Omega_1|_{m_1}$ and $|\Omega_2|_{m_2}$ respectively, then
\begin{equation*}
E(\Omega_1\times\Omega_2)=E(\Omega_1)E(\Omega_2),
\end{equation*}
where $\Omega_1\times\Omega_2$ is the Cartesian product in
$\R^{m_1+m_2}$.

To prove \eqref{e26} we let $B=\{x\in \R^2:|x|<1\}$. Then
\begin{equation*}
u_{B}(r,\theta)=J_0(j_0r), 0\le r<1,\, 0<\theta\leq 2\pi,
\end{equation*}
and
\begin{equation*}
\| u_{B} \|_1=\int_{[0,1]}dr\, r \int_{[0,2\pi)}d\theta\,J_0(j_0r)\approx 0.215882 (2\pi).
\end{equation*}
Since $\| u_{B} \|_{\infty}=J_0(0)=1$, we have that
\begin{equation*}
E(B)\approx 0.6782\frac{2}{\pi}.
\end{equation*}
\hspace*{\fill }$\square $

\section*{Acknowledgements}
The authors acknowledge support by the Leverhulme Trust through Emeritus Fellowship EM-2018-011-9, by GNAMPA of INdAM, and by a MIUR-PRIN 2017 grant ``Qualitative and quantitative aspects of nonlinear PDE's''. Michiel van den Berg wishes to thank Thomas Kappeler for helpful references to the literature.

\end{document}